\documentclass[reqno, 11pt]{amsart}
\usepackage{amssymb, amsmath, latexsym,enumerate,color}
\usepackage{graphicx,mathrsfs, booktabs}
\usepackage{xcolor}
\usepackage{float, leftindex, ragged2e}
\numberwithin{equation}{section}

\newcommand{\betaqa}{\begin{eqnarray}}
\newcommand{\eeqa}{\end{eqnarray}}
\newcommand{\betaqn}{\begin{eqnarray}}
\newcommand{\eeqn}{\end{eqnarray}}

\newcommand{\blr}{\begin{list}{$($\roman{cnt1}$)$}
 {\usecounter{cnt1} \setlength{\topsep}{0pt}
 \setlength{\itemsep}{0pt}}}
\newcommand{\bla}{\begin{list}{$($\alph{cnt2}$)$}
 {\usecounter{cnt2} \setlength{\topsep}{0pt}
 \setlength{\itemsep}{0pt}}}
\newcommand{\bln}{\begin{list}{$($\arabic{cnt3}$)$}
 {\usecounter{cnt3} \setlength{\topsep}{0pt}
 \setlength{\itemsep}{0pt}}}
\newcommand{\el}{\end{list}}
\newtheorem{theorem}{Theorem}[section]
\newtheorem{lemma}[theorem]{Lemma}

\newtheorem{example}[theorem]{Example}

\newtheorem{definition}[theorem]{Definition}
\newtheorem{proposition}[theorem]{Proposition}

\newcommand{\Rem}{\begin{rem} \rm}
\newcommand{\bdfn}{\begin{Def} \rm}
\newcommand{\edfn}{\end{Def}}

\newcommand{\ba}{\begin{array}}
\newcommand{\ea}{\end{array}}

\begin{document}

\begin{center}{\Large {A New Fractal Derivative and Its Integral Transform Methods for Economic Models}}
		\vspace{0.05cm}
		
	  Krishna Mani Nath$^1$, Bipan Hazarika$^{2, \ast}$  %\footnote{Corresponding author} 

$^{1,2}$Department of Mathematics, Gauhati University, Guwahati, Assam, India\\
	
{Email:  $^1$kmn.math@gmail.com; $^2$bh\_rgu@yahoo.com}
	
\end{center}
	\title{}
	\author{}
	\thanks{$^{\ast}$The corresponding author}
	\thanks{\today} 
	%\end{center} 
	\begin{abstract} 
         In this paper, we introduce the RLC-fractal derivative along with the fractal Laplace and fractal Sumudu transforms of it. Then we apply these transformations to solve economic models involving the RLC-fractal derivative.\\
         {\bf Keywords.}  Caputo fractal derivatives; fractal differential equation; Riemann-Liouville fractal derivative\\
        {\bf Mathematics Subject Classification:} 28A80, 26A33.
	\end{abstract}

	\maketitle
	
	%\maketitle

	\pagestyle{myheadings}
	\markboth{\rightline {\scriptsize {\textbf{\it K.M. Nath, B. Hazarika} }}}
	{\leftline{ \scriptsize{ \bf \it A New Fractal Derivative and Its Integral Transform Methods for Economic Models}}}	
	\maketitle

    %\tableofcontents

\section{Introduction}
Recently, researchers have recognized that many natural phenomena exhibit not only memory effects but also underlying fractal geometrical structures \cite{Barnsley,Mandelbrot,Parvate2003,Turner}. Numerous physical processes occurring in porous media, heterogeneous materials, biological tissues, and irregular surfaces evolve on domains with fractal characteristics. This observation has motivated the development of fractal differential equations, in which differentiation is defined with respect to a fractal measure or fractal dimension rather than the conventional Euclidean variable. Fractal differential operators provide an effective mathematical framework for describing dynamics on non-smooth geometries and irregular time scales, thereby extending the scope of classical and fractional calculus to systems with self-similar structures.

The integration of fractal geometry with fractional calculus has led to a new class of operators capable of simultaneously capturing memory effects and fractal characteristics. Among these, the Caputo fractal derivative has emerged as a particularly useful operator for modeling nonlocal phenomena governed by fractal differential equations. It is defined by applying a fractal integral operator to the fractal derivative $D^\alpha_Ff(S^\alpha_F(t))$ of a fractal differentiable function $f(S^\alpha_F(t))$, thereby combining the nonlocal features of fractional calculus with the geometric complexity of fractal spaces.

One of the fundamental concepts in economics is competitive equilibrium, which is achieved when the quantity of a commodity demanded by consumers equals the quantity supplied by producers. Classical equilibrium models generally assume that market dynamics evolve in smooth Euclidean spaces and that agents' interactions are governed by ordinary or fractional differential equations. However, many real-world economic systems exhibit irregular, heterogeneous, and self-similar features that cannot be adequately represented within these conventional frameworks. This motivates the use of fractal differential equations for economic modeling, as they provide a natural mathematical framework for incorporating both memory effects and fractal characteristics into market dynamics. Such an approach enables a more realistic description of complex economic processes and extends the applicability of equilibrium models to systems with non-smooth temporal or spatial structures.
%Recently, researchers have recognized that many natural phenomena possess not only memory effects but also fractal geometrical structures \cite{Barnsley,Mandelbrot,Parvate2003,Turner}. Physical processes occurring in porous media, heterogeneous materials, biological tissues, and irregular surfaces evolve in domains that exhibit fractal characteristics. This observation has motivated the development of fractal differential equations, where differentiation is defined with respect to a fractal measure or fractal dimension rather than the ordinary Euclidean variable. Fractal differential operators provide an effective mathematical framework for modeling dynamics on non-smooth geometries and irregular time scales, thereby extending the applicability of classical and fractional calculus to systems with self-similar structures.

%The combination of fractal geometry with fractional calculus has led to a new generation of operators capable of simultaneously describing memory effects and fractal properties. The Caputo fractal derivative has been one of the most useful operators for modeling non-local behaviors by fractal differential equations.  It is defined for a fractal differentiable function $f(S^\alpha_F(t))$, by a fractal integral operator applied to the derivative $D^\alpha_Ff(S^\alpha_F(t))$.

%One of the fundamental concepts in economics is competitive equilibrium, which occurs when the quantity of a commodity that consumers are willing to purchase is equal to the quantity that producers are prepared to supply. 

Mathematically, this equilibrium is characterized by the demand function $q_d$ and the supply function $q_s$, where $q_d$ denotes the quantity demanded and $q_s$ denotes the quantity supplied. Thus, when the expectations of the agents in the market are not considered, we have \begin{eqnarray}\label{i1}
    q_d(t)=d_0-d_1p(t),\;\;q_s(t)=-s_0+s_1p(t),
\end{eqnarray} where $p$ denotes the market price of the product, while $d_0$, $s_0$, $d_1$, and $s_1$ are constants that determine the levels of demand and supply. The equilibrium price is determined from the condition $$q_d(t)=q_s(t)\implies p^\ast=\dfrac{d_0+s_0}{d_1+s_1}.$$ Under this equilibrium condition, consumers demand exactly the same quantity that producers supply, resulting in a balanced market without shortage or surplus.

We consider the following simple price adjustment equation \cite{Nagle2011} \begin{align}\label{i2}
    \dfrac{dp}{dt}=\lambda(q_d-q_s),
\end{align} where $\lambda>0$. Then we get $$\dfrac{dp}{dt}+\lambda(d_1+s_1)p(t)=\lambda(d_0+s_0).$$ By solving \eqref{i2}, we obtain \begin{align*}
    p(t)=\dfrac{d_0+s_0}{d_1+s_1}-\left(p(0)+\dfrac{d_0+s_0}{d_1+s_1}\right)e^{-\lambda(d_1+s_1)t},
\end{align*} where $p(0)$ is the market price at time $t=0$. 

If we consider the expectations of the agents, \eqref{i1} changes to the following form \begin{align*}
    q_d(t)=d_0-d_1p(t)+d_2p'(t),\quad q_s(t)=-s_0+s_1p(t)-s_2 p'(t),
\end{align*} in which two additional factors $d_2,s_2\geqslant 0$ are added. For equilibrium, we have \begin{align*}
    q_d(t)=q_s(t)\implies p'(t)-\dfrac{d_1+s_1}{d_2+s_2}p(t)=-\dfrac{d_0+s_0}{d_2+s_2}
\end{align*}
By solving it, we get \begin{align*}
    p(t)=\dfrac{d_0+s_0}{d_1+s_1}-\left[p(0)+\dfrac{d_0+s_0}{d_1+s_1}\right]e^{\dfrac{d_1+s_1}{d_2+s_2}t}.
\end{align*}

We solve these types of economic models involving the RLC-fractal derivative which has been introduced in this paper. Fractal Sumudu transform has been used to solve them.

\section{Preliminaries to Fractal Calculus}

\begin{flushleft}
    {\bf Notation}: We  use $C_F^\alpha[a,b]$ to denote the family of $\alpha$-order fractal differentiable functions on $[a,b]$.
\end{flushleft}
\begin{lemma}\cite{Ali2,NHK2026}\label{rsl0}
        Let $\mathscr{S}(S_F^\alpha(v))$ and $\mathfrak{L}(S_F^\alpha(v))$ denote the fractal Sumudu tranform and the fractal Laplace tranform of a function $f(S_F^\alpha(t))$, respectively. Then \begin{eqnarray*}
            \mathfrak{L}[f](S_F^\alpha(v))=\dfrac{1}{S_F^\alpha(v)}\mathscr{S}[f]\left(\dfrac{1}{S_F^\alpha(v)}\right).
        \end{eqnarray*} Equivalently, \begin{eqnarray*}
            \mathscr{S}[f](S_F^\alpha(v))=\dfrac{1}{S_F^\alpha(v)}\mathfrak{L}[f]\left(\dfrac{1}{S_F^\alpha(v)}\right).
        \end{eqnarray*}
    \end{lemma}
\begin{definition}\cite{Ali3}
    If $f\in C_F^\alpha[a,b]$ and $\beta>0$, then the Riemann-Liouville fractal integral of order $\beta$ is defined as \begin{eqnarray*}
        ^{}_a{\mathcal{I}}_x^\beta f(S_F^\alpha(x)):=\dfrac{1}{\Gamma_F^\alpha(\beta)}\int_{S_F^\alpha(a)}^{S_F^\alpha(x)}\dfrac{f(S_F^\alpha(t))}{(S_F^\alpha(x)-S_F^\alpha(t))^{\alpha-\beta}}d_F^\alpha t,\;\;S_F^\alpha(x)>S_F^\alpha(a).
    \end{eqnarray*} 
\end{definition}
\begin{theorem} 
    Let $f\in C_F^\alpha[a,b]$ and $\beta>0$, then the fractal Laplace transform of $^{}_a{\mathcal{I}}_x^\beta f(S_F^\alpha(x))$ is \cite{Ali4} \begin{eqnarray}\label{rli1}
            \mathfrak{L}\left[^{}_0{\mathcal{I}}_x^\beta f(S_F^\alpha(x))\right]=\dfrac{\mathfrak{L}[f(S_F^\alpha(x))]}{S_F^\alpha(v)^\beta}.
        \end{eqnarray} and the corresponding fractal Sumudu transform is \cite{NHK2026}
        \begin{eqnarray*}          \mathscr{S}\left[^{}_0{\mathcal{I}}_x^\beta f(S_F^\alpha(x))\right]=S_F^\alpha(v)^\beta\mathscr{S}[f](S_F^\alpha(v)).
        \end{eqnarray*}
\end{theorem}
\begin{definition}\cite{Ali3}
    The fractal Mittag-Leffler function with one parameter is \begin{eqnarray*}        E_F^\alpha(S_F^\alpha(t))=\sum_{j=0}^\infty\dfrac{S_F^\alpha(t)^j}{\Gamma_F^\alpha(\eta j+a)},\;\eta>0,\;a\in\mathbb{R}.
    \end{eqnarray*}
\end{definition}
\begin{theorem} \cite{NHK2026}
    The fractal Sumudu transform of $E_F^\alpha(S_F^\alpha(t))$ is \begin{eqnarray*}
        \mathscr{S}\left[E_F^\alpha(S_F^\alpha(x))\right](S_F^\alpha(v))=\sum_{j=0}^\infty\dfrac{j!S_F^\alpha(v)^j}{\Gamma_F^\alpha(\gamma j+a)}.
    \end{eqnarray*}
\end{theorem}

\begin{definition}\cite{Ali3}
    Let $f\in C_F^{\alpha n}[a,b]$, then the Riemann-Liouville fractal derivative of order $\beta\in[0,1)$ is defined as \begin{eqnarray*}
        {^{}_a}\mathcal{D}_b^\beta f(S_F^\alpha(x)):=\dfrac{1}{\Gamma_F^\alpha(n-\beta)}(D_F^\alpha)^n\int_a^x\dfrac{f(S_F^\alpha(t))}{(S_F^\alpha(x)-S_F^\alpha(t))^{-n+\beta+\alpha}}d_F^\alpha t,
    \end{eqnarray*} where $n-\alpha\leqslant\beta<n.$
\end{definition}
\begin{theorem} \cite{NHK2026}
    Let $f\in C_F^{\alpha n}[a,b]$ and $\beta>0$, then the Sumudu transform of $^{}_0\mathcal{D}_x^\beta f(S_F^\alpha(x))$ is \begin{eqnarray*}
        \mathscr{S}\left[^{}_0\mathcal{D}_x^\beta f\right](S_F^\alpha(v))=\dfrac{1}{S_F^\alpha(v)^\beta}\mathscr{S}[f](S_F^\alpha(v))-\sum_{k=1}^{n}\dfrac{1}{S_F^\alpha(v)^{k}}\left[^{}_0\mathcal{D}_x^{\beta-k}f(S_F^\alpha(x))\right]_{x=0}.
    \end{eqnarray*}
\end{theorem}
\begin{definition}\label{def1}\cite{Ali3}
    Let $f\in C_F^{\alpha n}[a,b]$ and $\beta>0$, then Caputo fractal derivative of order $\beta$ is given by \begin{eqnarray*}
        ^{C}_a\mathcal{D}_x^\beta f(S_F^\alpha(x)):=\dfrac{1}{\Gamma_F^\alpha(n-\beta)}\int_{S_F^\alpha(a)}^{S_F^\alpha(x)} \left(S_F^\alpha(x)-S_F^\alpha(t)\right)^{n\alpha-\beta-\alpha}(D_F^\alpha)^nf(S_F^\alpha(t))d_F^\alpha t,
    \end{eqnarray*} where $n\alpha-\alpha\leqslant\beta<n\alpha$.
\end{definition}
\begin{theorem} 
    Let $f\in C_F^{\alpha n}[a,b]$ and $\beta>0$. Then the Laplace tranform of $^{C}_0\mathcal{D}_x^\beta f(S_F^\alpha(x))$ is given by \cite{Ali3} \begin{eqnarray}\label{cd2}
            \mathfrak{L}\left[^{C}_0\mathcal{D}_x^\beta f\right](S_F^\alpha(v))=S_F^\alpha(v)^\beta\mathfrak{L}[f](S_F^\alpha(v))-\sum_{j=1}^{n}S_F^\alpha(v)^{\beta-j}\left[(D_F^\alpha)^{j-1}f(S_F^\alpha(x))\right]_{x=0},
        \end{eqnarray} and the corresponding  Sumudu transformation is given by \cite{Ali2,NHK2026} \begin{eqnarray*}
            \mathscr{S}\left[^{C}_0\mathcal{D}_x^\beta f\right](S_F^\alpha(v))=\dfrac{1}{S_F^\alpha(v)^\beta}\mathscr{S}[f](S_F^\alpha(v))-\sum_{j=0}^{n-1}\dfrac{1}{S_F^\alpha(v)^{\beta-j}}\left[(D_F^\alpha)^{j}f(S_F^\alpha(x))\right]_{x=0}.
        \end{eqnarray*} 
\end{theorem}
\begin{definition} \cite{Ali5}
    Let $f,g:[0,\infty)\to\mathbb{R}$, then the convolution of $f$ and $g$ is given by \begin{eqnarray*}
        (f\ast g)(S_F^\alpha(t))=\int_{S_F^\alpha(0)}^{S_F^\alpha(t)}f\left(S_F^\alpha(t)-S_F^\alpha(z)\right)g\left(S_F^\alpha(z)\right)d_F^\alpha z.
    \end{eqnarray*} 
\end{definition}
\begin{theorem} \cite{NHK2026}
    Assume that $f,g:[0,\infty)\to\mathbb{R}$, then the Sumudu transformation of the convolution $f\ast g$ is \begin{eqnarray*}
        \mathscr{S}[f\ast g](S_F^\alpha(v))=S_F^\alpha(v)\,\mathscr{S}[f](S_F^\alpha(v))\,\mathscr{S}[g](S_F^\alpha(v)).
    \end{eqnarray*}
\end{theorem}    

\section{Combined Riemann-Liouville and Caputo (RLC) type fractal derivative}

We define a new type of fractal derivative by taking the linear combination of the Riemann-Liouville fractal integral and the Caputo fractal derivative.
\begin{definition}\label{def2}
    The combined Riemann-Liouville and Caputo fractal derivative (RLC-derivative) is defined as \begin{align*}
        ^{RLC}_0\mathcal{D}_t^\beta f(S^\alpha_F(t))&=L_1(\beta)^{}_a{\mathcal{I}}_t^{1-\beta} f(S^\alpha_F(t))+L_2(\beta)^{C}_0\mathcal{D}_t^\beta f(S^\alpha_F(t))\nonumber\\&=\dfrac{1}{\Gamma^\alpha_F(1-\beta)}\int_{S^\alpha_F(0)}^{s^\alpha_F(t)}\left(L_1(\beta)f(S^\alpha_F(y))\right.\nonumber\\&\left.+L_2(\beta)D^\alpha_Ff(S^\alpha_F(y))\right)(S^\alpha_F(t)-S^\alpha_F(y))^{-\beta}d^\alpha_Fy,%\label{rlc1}
    \end{align*} where $L_1$ and $L_2$ are functions depending on $\beta$ only. Equivalently, \begin{align*}
        ^{RLC}_0\mathcal{D}_t^\beta f(S^\alpha_F(t))&=^{}_a{\mathcal{I}}_t^{1-\beta}\left[L_1(\beta)f(S^\alpha_F(t))+L_2(\beta)D_F^\alpha f(S^\alpha_F(t))\right]\nonumber\\&=L_1(\beta)^{}_a{\mathcal{I}}_t^{1-\beta}f(S^\alpha_F(t))+L_2(\beta)^{C}_a\mathcal{D}_x^\beta f(S^\alpha_F(t)).
    \end{align*}
\end{definition}
\begin{theorem}
    The Laplace transform of $^{RLC}_0\mathcal{D}_x^\beta f(S^\alpha_F(t))$ is \begin{align*}
        \mathcal{L}[^{RLC}_0\mathcal{D}_t^\beta f](S^\alpha_F(v))=\left(\frac{L_1(\beta)}{S^\alpha_F(v)}+L_2(\beta)\right)S^\alpha_F(v)^\beta\mathcal{L}[f](S^\alpha_F(v))-L_2(\beta)S^\alpha_F(v)^{\beta-1}f(0),
    \end{align*} and the corresponding Sumudu transform is \begin{align*}
        \mathscr{S}[^{RLC}_0\mathcal{D}_t^\beta f](S^\alpha_F(v))=\dfrac{S^\alpha_F(v)L_1(\beta)+L_2(\beta)}{S^\alpha_F(v)^\beta}\mathscr{S}[f](S^\alpha_F(v))-\dfrac{L_2(\beta)}{S^\alpha_F(v)^\beta}f(0).
    \end{align*} 
\end{theorem}

\begin{proof}
    Using (\ref{rli1}) and (\ref{cd2}) for $n=1$, we get \begin{align}\label{rlc4}
        \mathcal{L}[^{RLC}_0\mathcal{D}_t^\beta f](S^\alpha_F(v))&=L_1(\beta)\dfrac{\mathfrak{L}[f]}{S_F^\alpha(v)^\beta}+L_2(\beta)\left(S^\alpha_F(v)^\beta\mathcal{L}[f](S^\alpha_F(v)-S^\alpha_F(v)^{\beta-1}f(0)\right)\nonumber\\&=\bigg[L_1(\beta)S^\alpha_F(v)^{\beta-1}+L_2(\beta)S^\alpha_F(v)^\beta\bigg]\mathfrak{L}[f]-L_2(\beta)S^\alpha_F(v)^{\beta-1}f(0)\nonumber\\&=\bigg[\dfrac{L_1(\beta)}{S^\alpha_F(v)}+L_2(\beta)\bigg]S^\alpha_F(v)^\beta\mathfrak{L}[f]-L_2(\beta)S^\alpha_F(v)^{\beta-1}f(0).
    \end{align}
    We get the Sumudu transform using Lemma \ref{rsl0} in (\ref{rlc4}).
\end{proof}

\begin{proposition}
    The inverse of the RLC-derivative operator $^{RLC}_0\mathcal{D}_t^\beta$ is \begin{align*} 
        {^{RLC}_0}I_t^\beta f(S^\alpha_F(t))=\dfrac{1}{L_2(\beta)}\int_{S^\alpha_F(0)}^{s^\alpha_F(t)}e^{-\frac{L_1(\beta)}{L_2(\beta)}(S^\alpha_F(t)-S^\alpha_F(v))} {^{}_0}\mathcal{D}_v^{1-\beta} f(S^\alpha_F(v))d^\alpha_Fv,
    \end{align*} which satisfies the relations \begin{align*}
        {^{RLC}_0} \mathcal{D}_t^\beta {^{RLC}_0} I_t^\beta f(S^\alpha_F(t))=f(S^\alpha_F(t))-\dfrac{S^\alpha_F(t)^{-\beta}}{\Gamma^\alpha_F(1-\beta)}F\_\lim\limits_{S^\alpha_F(t)\to 0} {^{}_0} I_t^{\beta} f(S^\alpha_F(t))\end{align*}
        and \begin{align*}%\label{rlc7}
        {^{RLC}_0} I_t^\beta {^{RLC}_0}\mathcal{D}_t^\beta f(S^\alpha_F(t))=f(S^\alpha_F(t))-e^{-\frac{L_1(\beta)}{L_2(\beta)}S^\alpha_F(t)}f(0).
    \end{align*} 
\end{proposition}

\section{Equations containing fractal RLC-derivative operator and the solutions to them}

In this section, we solve some equations involving fractal RLC-derivation operator using the fractal Laplace and Sumudu transforms.
\begin{example}
    Consider the following fractal differential equation \begin{align*}%\label{rlc8}
        {^{RLC}_0} \mathcal{D}_t^\beta f(S^\alpha_F(t))=0,\quad f(0)=M.
    \end{align*} Applying fractal Laplace transform and using $f(0)=M$, we obtain
    \begin{align*} 
        &\left(\frac{L_1(\beta)}{S^\alpha_F(v)}+L_2(\beta)\right)S^\alpha_F(v)^\beta\mathcal{L}[f](S^\alpha_F(v))-L_2(\beta)S^\alpha_F(v)^{\beta-1}M=0\nonumber\\\implies& \mathcal{L}[f](S^\alpha_F(v))=\dfrac{L_2(\beta)S^\alpha_F(v)^{\beta-1}M}{\left(\frac{L_1(\beta)}{S^\alpha_F(v)}+L_2(\beta)\right)S^\alpha_F(v)^\beta}=\dfrac{M}{S^\alpha_F(v)+\dfrac{L_1(\beta)}{L_2(\beta)}}.
    \end{align*} Taking the inverse Laplace transform, we get \begin{align*}
            f(S^\alpha_F(t))=Me^{-\frac{L_1(\beta)}{L_2(\beta)}S^\alpha_F(t)}. 
    \end{align*}
\end{example}

\begin{example}
    Consider the following fractal derivative for an arbitrary constant $\lambda$, \begin{align*}%\label{rlc11}
        {^{RLC}_0} \mathcal{D}_t^\beta f(S^\alpha_F(t))=\lambda f(S^\alpha_F(t)),\quad f(0)=1.
    \end{align*} Applying fractal Sumudu transform, we get \begin{align*}
        &\dfrac{S^\alpha_F(v)L_1(\beta)+L_2(\beta)}{S^\alpha_F(v)^\beta}\mathscr{S}[f](S^\alpha_F(v))-\dfrac{L_2(\beta)}{S^\alpha_F(v)^\beta}f(0)=\lambda\mathscr{S}[f](S^\alpha_F(v))\nonumber\\ \implies&\mathscr{S}[f](S^\alpha_F(v))=\dfrac{L_2(\beta)}{S^\alpha_F(v)L_1(\beta)+L_2(\beta)-\lambda S^\alpha_F(v)^\beta}\nonumber\\& \hspace{2.13cm}=\left(1-\dfrac{\lambda S^\alpha_F(v)^\beta-S^\alpha_F(v)L_1(\beta)}{L_2(\beta)}\right)^{-1}\nonumber\\& \hspace{2.13cm}=\sum_{n=0}^{\infty}\left(\dfrac{\lambda S^\alpha_F(v)^\beta-S^\alpha_F(v)L_1(\beta)}{L_2(\beta)}\right)^n\nonumber\\& \hspace{2.13cm}=\sum_{n=0}^{\infty}\dfrac{1}{L_2(\beta)^n}\sum_{k=0}^{n}\binom{n}{k}\left(\lambda S^\alpha_F(v)^\beta\right)^{n-k}(-S^\alpha_F(v)L_1(\beta))^k\nonumber\\& \hspace{2.13cm}=\sum_{n=0}^{\infty}\sum_{k=0}^{n}\binom{n}{k}\dfrac{(-L_1(\beta))^n\lambda^{n-k}}{L_2(\beta)^n}S^\alpha_F(v)^{\beta n-\beta k+k}.%\label{rlc12}
    \end{align*}
    Taking the inverse Sumudu transform for each $n$, we obtain \begin{align*}        f(S^\alpha_F(t))=\sum_{n=0}^{\infty}\sum_{k=0}^{n}\binom{n}{k}\dfrac{(-L_1(\beta))^n\lambda^{n-k}}{L_2(\beta)^n}\cdot\dfrac{S^\alpha_F(t)^{\beta n-\beta k+k}}{\Gamma_F^\alpha(\beta n-\beta k+k+1)}.%\label{rlc13} 
    \end{align*}
\end{example}

\section{Economic model formulated via the RLC fractal derivative}

\begin{enumerate}[\text{Case }1.]
    \item The price adjustment equation (see \cite{Cohen}),  formulated using the RLC fractal derivative without the expectations of agents, is given as follows
    \begin{align}\label{rlc14}  ^{RLC}\mathfrak{D}^\beta_Fp(S^\alpha_F(t))+k(d_1+s_1)p(S^\alpha_F(t))=k(d_0+s_0).
    \end{align} 
    Applying fractal Sumudu transform, we have 
        \begin{align*}
        &\dfrac{S^\alpha_F(v)L_1(\beta)+L_2(\beta)}{S^\alpha_F(v)^\beta}\mathscr{S}[p](S^\alpha_F(v))-\dfrac{L_2(\beta)}{S^\alpha_F(v)^\beta}p(0)+k(d_1+s_1)\mathscr{S}[p](S^\alpha_F(v))=k(d_0+s_0)\nonumber\\ \implies&\mathscr{S}[p](S^\alpha_F(v))=\dfrac{L_2(\beta)p(0)+k(d_0+s_0)S^\alpha_F(v)^\beta}{S^\alpha_F(v)L_1(\beta)+L_2(\beta)+k(d_1+s_1)S^\alpha_F(v)^\beta}\nonumber\\ & \hspace{2.15cm}=\dfrac{d_0+s_0}{d_1+s_1}\left[1-\dfrac{-L_2(\beta)-S^\alpha_F(v)L_1(\beta)}{k(d_1+s_1)S^\alpha_F(v)^\beta}\right]^{-1}\nonumber\\ & \hspace{2.15cm}+p(0)\left[1-\dfrac{-S^\alpha_F(v)L_1(\beta)-k(d_1+s_1)S^\alpha_F(v)^\beta}{L_2(\beta)}\right]^{-1}\nonumber\\ & \hspace{2.15cm}=\dfrac{d_0+s_0}{d_1+s_1}\sum_{j=0}^{\infty}\left[\dfrac{-L_2(\beta)-S^\alpha_F(v)L_1(\beta)}{k(d_1+s_1)S^\alpha_F(v)^\beta}\right]^j\nonumber\\ & \hspace{2.15cm}+p(0)\sum_{j=0}^{\infty}\left[\dfrac{-S^\alpha_F(v)L_1(\beta)-k(d_1+s_1)S^\alpha_F(v)^\beta}{L_2(\beta)}\right]^j\nonumber\\ & \hspace{2.15cm}=\dfrac{d_0+s_0}{d_1+s_1}\sum_{j=0}^{\infty}\sum_{r=0}^{j}(-1)^j\binom{j}{r}\dfrac{L_1(\beta)^{j-r}L_2(\beta)^r}{k^j(d_1+s_1)^j}S^\alpha_F(v)^{(1-\beta)j-r}\nonumber\\ & \hspace{2.15cm}+p(0)\sum_{j=0}^{\infty}\sum_{r=0}^{j}(-1)^j\binom{j}{r}\dfrac{L_1(\beta)^{j-r}k^r(d_1+s_1)^r}{L_2(\beta)^j}S^\alpha_F(v)^{j+(\alpha-1)r}.%\label{rlc14a}
    \end{align*}
    Taking the inverse Sumudu transform for each $n$, we obtain \begin{align*}
        &p(S^\alpha_F(t))=\dfrac{d_0+s_0}{d_1+s_1}\sum_{j=0}^{\infty}\sum_{r=0}^{j}(-1)^j\binom{j}{r}\dfrac{L_1(\beta)^{j-r}L_2(\beta)^r}{k^j(d_1+s_1)^j}\cdot\dfrac{S^\alpha_F(t)^{(1-\beta)j-r}}{\Gamma_F^\alpha((1-\beta)j-r+1)}\nonumber\\ & \hspace{2.15cm}+p(0)\sum_{j=0}^{\infty}\sum_{r=0}^{j}(-1)^j\binom{j}{r}\dfrac{L_1(\beta)^{j-r}k^r(d_1+s_1)^r}{L_2(\beta)^j}\cdot\dfrac{S^\alpha_F(t)^{j+(\alpha-1)r}}{\Gamma_F^\alpha(j+(\alpha-1)r+1)}.
    \end{align*}
    \item The price adjustment equation, formulated using the wsk-fractal derivative, when the expectations of agents are considered, is given by \begin{align*}
        ^{RLC}\mathfrak{D}^\beta_Fp(S_F^\alpha(t))-\dfrac{d_1+s_1}{d_2+s_2}p(S_F^\alpha(t))=-\dfrac{d_0+s_0}{d_2+s_2}.
    \end{align*} Applying Laplace Sumudu transform, we get \begin{align*}
        &\dfrac{S^\alpha_F(v)L_1(\beta)+L_2(\beta)}{S^\alpha_F(v)^\beta}\mathscr{S}[p](S^\alpha_F(v))-\dfrac{L_2(\beta)}{S^\alpha_F(v)^\beta}p(0)-\dfrac{d_1+s_1}{d_2+s_2}\mathscr{S}[p](S^\alpha_F(v))=-\dfrac{d_0+s_0}{d_2+s_2}\nonumber\\\implies&\mathscr{S}[p](S^\alpha_F(v))=\dfrac{L_2(\beta)p(0)(d_2+s_2)-(d_0+s_0)S^\alpha_F(v)^\beta}{S^\alpha_F(v)L_1(\beta)(d_2+s_2)+L_2(\beta)(d_2+s_2)-(d_1+s_1)S^\alpha_F(v)^\beta}\nonumber\end{align*}\begin{align*} & \hspace{2.15cm}=\dfrac{d_0+s_0}{d_1+s_1}\left[1-\dfrac{(S^\alpha_F(v)L_1(\beta)+L_2(\beta))(d_2+s_2)}{(d_1+s_1)S^\alpha_F(v)^\beta}\right]^{-1}\nonumber\\ & \hspace{2.15cm}+p(0)\left[1-\dfrac{(d_1+s_1)S^\alpha_F(v)^\beta-S^\alpha_F(v)L_1(\beta)(d_2+s_2)}{L_2(\beta)(d_2+s_2)}\right]^{-1}\nonumber\\ & \hspace{2.15cm}=\dfrac{d_0+s_0}{d_1+s_1}\sum_{j=0}^{\infty}\left[\dfrac{d_2+s_2}{(d_1+s_1)S^\alpha_F(v)^\beta}(S^\alpha_F(v)L_1(\beta)+L_2(\beta))\right]^j\nonumber\\ & \hspace{2.15cm}+p(0)\sum_{j=0}^{\infty}\left[\dfrac{(d_1+s_1)S^\alpha_F(v)^\beta-S^\alpha_F(v)L_1(\beta)(d_2+s_2)}{L_2(\beta)(d_2+s_2)}\right]^j\nonumber\\ & \hspace{2.15cm}=\dfrac{d_0+s_0}{d_1+s_1}\sum_{j=0}^{\infty}\sum_{r=0}^{j}\binom{j}{r}\left(\dfrac{d_2+s_2}{d_1+s_1}\right)^jL_1(\beta)^{j-r}L_2(\beta)^rS^\alpha_F(v)^{(1-\beta)j-r}\nonumber\\ & \hspace{2.15cm}+p(0)\sum_{j=0}^{\infty}\sum_{r=0}^{j}\binom{j}{r}\dfrac{(-L_1(\beta)(d_2+s_2))^{j-r}(d_1+s_1)^r}{L_2(\beta)^j(d_2+s_2)^j}S^\alpha_F(v)^{(\beta-1)j-r}.
    \end{align*}
    Taking the inverse Sumudu transform for each $n$, we obtain \begin{align}
        &p(S^\alpha_F(t))=\dfrac{d_0+s_0}{d_1+s_1}\sum_{j=0}^{\infty}\sum_{r=0}^{j}\binom{j}{r}\left(\dfrac{d_2+s_2}{d_1+s_1}\right)^j\dfrac{L_1(\beta)^{j-r}L_2(\beta)^rS^\alpha_F(t)^{(1-\beta)j-r}}{\Gamma_F^\alpha((1-\beta)j-r+1)}\nonumber\\ & \hspace{2.75cm}+p(0)\sum_{j=0}^{\infty}\sum_{r=0}^{j}\binom{j}{r}\dfrac{(-L_1(\beta)(d_2+s_2))^{j-r}(d_1+s_1)^rS^\alpha_F(t)^{j+(\alpha-1)r}}{L_2(\beta)^j(d_2+s_2)^j\,\Gamma_F^\alpha(j+(\alpha-1)r+1)}.\label{rlc18}\nonumber
    \end{align}    
\end{enumerate}

\section{Comparison of the solutions of different economic models}
%\textcolor{blue}{
In this section, we compare the solutions of the economic models formed by different fractal differential operators in two cases. In case I, the expectations of the agents are not considered. Meanwhile, we consider the expectations of the agents in case II.
%[\text{Case} I:]
\begin{enumerate}
    \item[Case I:] Here, we consider different solutions:
    \begin{itemize}
        \item The solution to the economic model \begin{eqnarray*}%\label{c1}
            ^{C}_0\mathcal{D}_x^\beta p(t)+\lambda(d_1+s_1)p(t)=\lambda(d_0+s_0),\;\beta\in(0,1),
        \end{eqnarray*} formulated using the Caputo fractal derivative is \cite{Ali6} \begin{equation}\label{c2}
            p(S^\alpha_F(t))=p(0)E_F^\alpha\left(-\lambda(d_1+s_1)S^\alpha_F(t)^\beta\right)+\dfrac{\lambda(d_0+s_0)}{d_1+s_1}\left(1-E_F^\alpha\left(-\lambda(d_1+s_1)S^\alpha_F(t)^\beta\right)\right),
        \end{equation}   
        where $\lambda>0$ and $E_F^\alpha(t)$ is the fractal Mittag-Leffler function with one parameter given by $$E_F^\alpha(t)=\sum_{j=0}^\infty\dfrac{S_F^\alpha(t)^j}{\Gamma_F^\alpha(\eta j+a)}, \mbox{~where~} p \mbox{~is ~an~} F^\alpha-\mbox{differential ~function~ on~} [a,b].$$ 
        \item The solution to the economic model
        \begin{eqnarray*}\label{c3}        ^{\text{wsk}}\mathfrak{D}^\gamma_Fp(S^\alpha_F(t))+k(d_1+s_1)p(S^\alpha_F(t))=k(d_0+s_0),
    \end{eqnarray*} formulated using the wsk-fractal derivative \cite{NHK2026}, is 
    \begin{eqnarray}
        p(S^\alpha_F(t))&=&\dfrac{\mathcal{N}(\gamma)p(0)}{\mathcal{N}(\gamma)+k(d_1+s_1)(1-\gamma)}e^{-\dfrac{k\gamma(d_1+s_1)S^\alpha_F(t)}{\mathcal{N}(\gamma)+k(1-\gamma)(d_1+s_1)}}\nonumber\\&&+\dfrac{d_0+s_0}{d_1+s_1}\left[1-e^{-\dfrac{\gamma k(d_1+s_1)}{\mathcal{N}(\gamma)+k(d_1+s_1)(1-\gamma)}S^\alpha_F(t)}\right],\;\gamma\in(0,1),\label{c4}
    \end{eqnarray} where $\mathcal{N}(\gamma)$ is a normalization function such that $\mathcal{N}(0)=\mathcal{N}(1)=1$.
        \item The solution to the economic model
        \begin{align}\label{c5}                  ^{RLC}\mathfrak{D}^\beta_Fp(S^\alpha_F(t))+k(d_1+s_1)p(S^\alpha_F(t))=k(d_0+s_0),
        \end{align} formulated using the RLC-derivative introduced in Section 5, is 
        \begin{align}
            p(S^\alpha_F(t))&=\dfrac{d_0+s_0}{d_1+s_1}\sum_{j=0}^{\infty}\sum_{r=0}^{j}(-1)^j\binom{j}{r}\dfrac{L_1(\beta)^{j-r}L_2(\beta)^r}{k^j(d_1+s_1)^j}\cdot\dfrac{S^\alpha_F(t)^{(1-\beta)j-r}}{\Gamma_F^\alpha((1-\beta)j-r+1)}\nonumber\\ & +p(0)\sum_{j=0}^{\infty}\sum_{r=0}^{j}(-1)^j\binom{j}{r}\dfrac{L_1(\beta)^{j-r}k^r(d_1+s_1)^r}{L_2(\beta)^j}\cdot\dfrac{S^\alpha_F(t)^{j+(\alpha-1)r}}{\Gamma_F^\alpha(j+(\alpha-1)r+1)},\label{c6}
        \end{align} where $L_1$ and $L_2$ are functions depending only on $\beta$.
    \end{itemize}
    The solution \eqref{c2} has  Mittag-Leffler type decay, not exponential decay. Its long-time behaviour is $$p(S^\alpha_F(t))\to\dfrac{\lambda(d_0+s_0)}{d_1+s_1}$$ whenever $E_F^\alpha\left(-\lambda(d_1+s_1)S^\alpha_F(t)^\beta\right)\to 0$.\\
    The solution \eqref{c4} does not contain the Mittag-Leffler type growth/decay seen in \eqref{c2}. Its behaviour is exponential. So it is not a direct special case of \eqref{c2} unless extra identifications are made.\\
    The solution \eqref{c6} reduces to a Mittag-Leffler type solution like \eqref{c2} when the parameters $L_1$, $L_2$ are chosen so that the double sums collapse to the single Mittag-Leffler series.\\
        %In general, these solutions are not equivalent. Each corresponds to a different level of approximation and  different parameter choices. To compare them numerically or analytically, one must specify the relationship among the parameters 
    In general, these solutions are not equivalent, as each corresponds to a different level of approximation and a distinct choice of parameters. Consequently, a meaningful numerical or analytical comparison of the solutions can only be made after specifying the relationships among the parameters $\beta$, $\gamma$, $\mathcal{N}(\gamma)$, $k$, $L_1$ and $L_2$.
    \item[Case-II:]  In this case as well, by comparing the solutions of the economic models, where the agents' expectations are explicitly taken into account, we conclude that the models are not equivalent. A meaningful numerical or analytical comparison of the solutions can only be carried out after specifying the relationships among the various parameters involved in the respective solutions. %In this case also, by comparing the solutions of the economic models (the expectations of the agents are considered here), we can conclude that they are not equivalent. We can compare them numerically or analytically only by specifying the relationship among different parameters present in the solutions.
\end{enumerate}
%}

\section{Conclusion}
%\textcolor{red}{
In this study, we introduced a new fractal derivative, namely the RLC-derivative, and developed the corresponding fractal Laplace and Sumudu transforms associated with this operator. As an application, an economic model was formulated and analyzed using the proposed derivative. Furthermore, a comparative study revealed that the solutions obtained using the Caputo fractal derivative, the wsk-fractal derivative, and the proposed RLC-derivative are not equivalent. Hence, the choice of the differential operators plays a crucial role in capturing different aspects of economic behavior and may provide new insights into the modeling of complex economic systems.%}

%\noindent  {\bf Funding:} Not Applicable, the research is not supported by any funding agency.\\
%{\bf Conflict of Interest/Competing interests:} The authors declare that they have no computing interests.\\
 %{\bf Availability of data and material:} The article does not contain any data for analysis.\\
 %{\bf Code Availability:} Not Applicable.\\
 %{\bf Author's Contributions:} All the authors have equal contribution for the preparation of the article. All authors read and approved the final manuscript.

%\newpage 
%\section*{Declaration}
%\noindent {\bf Availability of data and material:} The article does not contain any data for analysis.\\
%{\bf Code Availability:} Not Applicable.\\
%{\bf Author's Contributions:} All the authors have equal contributions for the preparation of the article. All authors read and approved the final manuscript.\\
%{\bf Conflict of Interest/Competing interests:} The authors declare that they have no computing interests. 

\section*{Acknowledgments}
The first author acknowledges support of the UGC Junior Research Fellowship funded under NTA Ref. No.: 231620032207.
%\newpage 


\begin{thebibliography}{00}

   \bibitem{Barnsley} M.F. Barnsley, {\it Fractals everywhere}, Academic press (2014).

\bibitem{Mandelbrot} B.B. Mandelbrot, {\it The fractal geometry of nature}, New York: WH freeman (1983). 173

\bibitem{Parvate2003} A. Parvate and A. Gangal,  Calculus on fractal subsets of real line-I: Formulation, {\it Fractals}, {\bf 17}(1) (2009) 53-81.


   \bibitem{Turner} M.J. Turner, {\it Modeling nature with fractals}. Leicester (2000).


 \bibitem{Nagle2011} R. K. Nagle, E. B. Saff, A. D. Snider, \emph{Fundamentals of Differential Equations}, 8th Edition, Pearson, London, 2011.

\bibitem{Ali2} A. K. Golmankhaneh, C. Tun\c{c},  Sumudu transform in fractal calculus, {\it Appl. Math. Comput.} {\bf 350} (2019) 386--401.

\bibitem{NHK2026} K. M. Nath, B. Hazarika,  H. Kalita,  (2026). Fractal Sumudu Transform and Economic Models. arXiv preprint arXiv:2602.17723.
   
\bibitem{Ali3}  A. K. Golmankhaneh, D. Baleanu,  New Derivatives on the Fractal Subset of Real-Line, {\it Entropy}, {\bf 18}(2), 1 (2016). https://doi.org/10.3390/e18020001.

   \bibitem{Ali4} A. K. Golmankhaneh, D. Baleanu, Non-local Integrals and Derivatives on Fractal Sets with Applications, {\it Open Physics},  {\bf 14}(1)(2016) 542--548. %https://doi.org/10.1515/phys-2016-0062


\bibitem{Ali5} A. K. Golmankhaneh, {\it Fractal Calculus and its Applications: $F^\alpha$-Calculus} (2022).  
   10.1142/12988.
   
   \bibitem{Cohen} D. Cohen-Vernik, A. Pazgal,  Price Adjustment Policy with Partial Refunds, {\it J. Retailing},  {\bf 93}(4)(2017) 507--526. %ISSN 0022-4359, https://doi.org/10.1016/j.jretai.2017.08.002.

 
   \bibitem{Ali6} A. K. Golmankhaneh, K. K. Ali, Resat Yilmazer, \& Mohammed K. A. Kaabar, Economic Models Involving Time Fractal, {\it J.  Math. Modeling Finance}, {\bf 1}(1)(2021) 131--146.


   

  
  
                    




   
   
   
   

   %\bibitem{Ali4}  Khalili Golmankhaneh, Alireza \& Jorgensen, Palle \& Serpa, Cristina \& Welch, Kerri. (2024). About Sobolev spaces on fractals: fractal gradians and Laplacians. Aequationes mathematicae. 99. 465-490. 10.1007/s00010-024-01060-6. 
\end{thebibliography}
 \end{document}